\documentclass[11pt,a4paper]{article}
\usepackage[T1]{fontenc}
\usepackage{lmodern}
\usepackage{amsmath,amssymb,amsthm,mathtools}
\usepackage{microtype}
\usepackage[numbers,sort&compress]{natbib}
\usepackage[a4paper,margin=22mm]{geometry}
\usepackage[hidelinks]{hyperref}
\usepackage{needspace}
\hypersetup
{
	colorlinks=true,
	linkcolor=blue,
	filecolor=blue,
	urlcolor=blue,
	citecolor=blue,
}
\usepackage{setspace}
\newtheorem{theorem}{Theorem}[section]
\newtheorem{lemma}[theorem]{Lemma}

\theoremstyle{definition}
\newtheorem{claim}{Claim}
\newtheorem{conjecture}{Conjecture}

\theoremstyle{plain}

\newcommand{\odd}{\operatorname{odd}}
\newcommand{\one}{\mathbf{1}}
\newcommand{\dd}{\mathsf{D}}
\newcommand{\vol}{\operatorname{vol}}

\title{Optimal path decompositions of sparse random graphs}
\author{Jifu Lin\thanks{Email: \href{mailto:jifulin01@163.com}{jifulin01@163.com}.}
\and Heng Yang\thanks{Corresponding author. Email:
\href{mailto:younghmath@163.com}{younghmath@163.com}.}}
\date{}

\begin{document}
\maketitle
\vspace{-1.5em}
\begin{center}\small
Department of Mathematics, East China Normal University, Shanghai 200241, China
\end{center}

\begin{abstract}
Let $\odd(G)$ be the number of odd-degree vertices of a graph $G$.
We prove that, if $p=o(1)$ and $np/\log n\to\infty$, then with
probability tending to one the random graph $G(n,p)$ has an edge
decomposition into exactly $\odd(G)/2$ paths. This confirms a conjecture of Glock, K\"uhn and Osthus from 2016. The proof combines a local pairing argument based on the lopsided Lovász local lemma
with cycle decompositions of random graphs.
\end{abstract}



\section{Introduction}

Path and cycle decompositions are among the most classical and fundamental topics in structural graph theory, closely tied to several long-standing conjectures.

A path decomposition (resp. cycle decomposition) of a graph $G$ is a family of nonempty paths (resp. cycles) whose edge sets partition $E(G)$.
Let $\odd(G)$ denote the number of vertices of odd degree in $G$. Every path decomposition of $G$ contains at least $\odd(G)/2$ paths, since every odd-degree vertex must be an endpoint of at least one path, and each path has exactly two endpoints. Thus, it is natural to ask when this trivial lower bound can be attained.

Gallai conjectured that every connected graph on $n$ vertices admits a path decomposition into at most $\lceil n/2\rceil$ paths. 
Lov\'asz~\cite{Lovasz68} proved several foundational decomposition results, including an upper bound of $n-1$ paths for arbitrary graphs. 
Subsequent work improved the general upper bound to $\lfloor 2n/3\rfloor$ paths~\cite{DK00,Yan98}. For dense graphs with fixed positive linear minimum degree, Gir\~ao, Granet, K\"uhn and Osthus~\cite{GGKO21} obtained a path decomposition upper bound of $n/2+o(n)$, which is asymptotically near the bound in Gallai's conjecture. Gallai's conjecture remains open in general, which has motivated extensive study of optimal path decompositions for well-studied graph classes.

The Erd\H{o}s--Gallai conjecture~\cite{EG-conj,EGP66} states that the edge set of any $n$-vertex graph can be partitioned into $O(n)$ cycles and edges. 
Conlon, Fox and Sudakov~\cite{CFS14} proved an $O(n\log\log n)$ upper bound for general graphs, and established the conjecture for random graphs and for graphs with linear minimum degree. A resolution of the Erdős–Gallai conjecture was recently given by Kim~\cite{Kim26}; OpenAI~\cite{OpenAI26} obtained an independent proof.

For random graphs, Kor\'andi, Krivelevich and Sudakov~\cite{KKS15} carried out a more systematic study of the problem: for a large range of $p$, $G(n,p)$ with high probability can be decomposed into $n/4 + np/2 + o(n)$ cycles and edges, which is asymptotically best possible. Glock, K\"uhn and Osthus~\cite[Theorem~1.2(ii)]{GKO16} proved that for every fixed $p\in(0,1)$, $G(n,p)$ with high probability admits an optimal path decomposition into exactly $\max\{\odd(G)/2,\lceil\Delta(G)/2\rceil\}$
paths, matching the natural lower bound exactly. They also conjectured that the
lower bound $\odd(G)/2$ is attainable in the sparse range:
\begin{conjecture}[Glock, K\"uhn and Osthus~\cite{GKO16}]\label{conj:GKO}
	Suppose $p=o(1)$ and
	$np/\log n\to\infty$. With probability tending to one as
	$n\to\infty$, $G\sim G(n,p)$ can be decomposed into $\odd(G)/2$ paths.
\end{conjecture}

In this paper, we prove Conjecture~\ref{conj:GKO}.

\begin{theorem}\label{thm:main}
Suppose	 $p=o(1)$ and $np/\log n\to\infty$.
With probability tending to one as $n\to\infty$, the random graph
$G\sim G(n,p)$ has a decomposition into exactly $\odd(G)/2$ paths.
\end{theorem}

Our result achieves the trivial lower bound $\odd(G)/2$ and thus gives an optimal path decomposition for sparse random graphs in this range.

The proof combines two ingredients. A key lemma gives an optimal
path decomposition for sparse, nearly regular graphs with suitable
common-neighbor and spectral bounds. For graphs of larger average degree, we rely on the cycle
decompositions of Kor\'andi, Krivelevich and Sudakov~\cite{KKS15}. Using a small set of reserve edges, we replace each cycle by two paths whose endpoints are odd-degree vertices and are all distinct
across the entire construction. We then apply our key lemma to the edges that remain unused.

Section~\ref{sec:prelim} collects notation and auxiliary results.
Section~\ref{sec:pairing} proves the key lemma.
Section~\ref{sec:main} gives the constructions that replace cycles
by paths and proves Theorem~\ref{thm:main}.

\section{Preliminaries}\label{sec:prelim}
All graphs are finite, simple and undirected. We use the standard terminology in~\cite[Chapter~1]{Bondy2008}. A \emph{walk} is a sequence $v_0e_1v_1\cdots e_kv_k$ in which $e_i=v_{i-1}v_i$ for
$1\le i\le k$. A \emph{trail} is a walk with distinct edges, and a \emph{path}
is a walk with distinct vertices. A \emph{cycle} is a closed walk of length
at least three whose vertices are distinct apart from
$v_0=v_k$. The length of a walk is its number of edges. All paths
in a decomposition are nonempty.

For a graph $H$, we write $V(H)$ and $E(H)$ for its vertex and edge sets, respectively. For a vertex $v \in V(H)$, we denote by $N_H(v)$ its neighborhood and by $d_H(v) = |N_H(v)|$ its degree. For $U\subseteq V(H)$, let $d_H(v,U)=|N_H(v)\cap U|$,
and let $e_H(U)$ count the edges with both ends in $U$. The minimum
and maximum degrees are $\delta(H)$ and $\Delta(H)$. Define
$D(H)=\{v:d_H(v)\text{ is odd}\}$ and $\odd(H)=|D(H)|$. Also, $\Delta_2(H)=\max_{u\ne v}|N_H(u)\cap N_H(v)|$ is the largest
number of common neighbors of two distinct vertices. A graph is
\emph{even} if every degree is even; it need not be connected.
Every even graph has a cycle decomposition~\cite[Theorem~2.7]{Bondy2008}.
For graphs $H_1,H_2$ on the same vertex set $V$, we use
$V(H_1\cup H_2)=V$ and
$E(H_1\cup H_2)=E(H_1)\cup E(H_2)$.
We write $H_1\dot\cup H_2$ when
$E(H_1)\cap E(H_2)=\varnothing$. For a subgraph $F$ of $H$,
$V(H-F)=V(H)$ and $E(H-F)=E(H)\setminus E(F)$. We use standard asymptotic notation $O$, $\Omega$, $o$ and $\gg$. A property holds \emph{with high probability} (whp) if its probability tends to $1$ as $n\to\infty$.

In $G(n,p)$, each unordered pair of distinct vertices is an edge
independently with probability $p$. We use the following Chernoff
bounds~\cite[Theorems~A.1.12--A.1.13]{AS16}. If $Y$ is a sum of independent
zero-one random variables with mean $\mu>0$, then
\[
 \Pr(|Y-\mu|\ge\eta\mu)\le2e^{-\eta^2\mu/3}
 \quad(0<\eta\le1),\qquad
 \Pr(Y\ge t)\le(e\mu/t)^t\quad(t\ge\mu).
\]
We also use the one-sided lower-tail bound
$\Pr(Y\le(1-\eta)\mu)\le e^{-\eta^2\mu/2}$ for $0<\eta<1$.
For any finite collection of events $\mathcal A_i$, we use
$\Pr(\bigcup_i\mathcal A_i)\le\sum_i\Pr(\mathcal A_i)$,
which requires no independence. For
$Y\sim\operatorname{Bin}(m,s)$, we also use
$\Pr(Y\text{ is odd})=(1-(1-2s)^m)/2$;
see~\cite[Section~2]{KKS15}.

We need a few matrix estimates to count walks. The adjacency
matrix $A_H$ has entry one at $(u,v)$ when $uv\in E(H)$ and zero
otherwise. If $H$ has no isolated vertices, let $\dd_H$ be its
diagonal matrix of degrees and define
$P_H=\dd_H^{-1}A_H$ and $S_H=\dd_H^{-1/2}A_H\dd_H^{-1/2}$.
Thus $(P_H^k)_{vu}$ is the probability that a walk starting at $v$
and choosing a neighbor uniformly at each step is at $u$ after
$k$ steps. The symmetric matrix $S_H$ has eigenvalue one with
eigenvector $(\sqrt{d_H(v)}:v\in V(H))$. Let $\lambda(H)$ be the
largest absolute value of its remaining eigenvalues, after removing
one copy of the eigenvalue one. In particular, $\lambda(H)<1$
implies that $H$ is connected.

For a real vector $x$, let $\|x\|_2 = (\sum_i x_i^2)^{1/2}$. For a real matrix $M$, let $\|M\|$ be its spectral norm and $\|M\|_\infty = \max_i \sum_j |M_{ij}|$ its maximum row sum norm. For symmetric $M$, $\|M\|$ equals the largest absolute eigenvalue.
The vector
$\one$ has all entries one, and $\one_U$ has entry one on $U$ and
zero elsewhere. For any graph $F$, $\|A_F\| \le \Delta(F)$.

\begin{lemma}\label{lem:normalize}
If $d_H(v)=(1+o(1))r$ uniformly over $v$ and
$\|A_H-(r/n)\one\one^{\mathsf T}\|=o(r)$, then
$\lambda(H)=o(1)$.
\end{lemma}
\begin{proof}
Let $E=A_H-(r/n)\one\one^{\mathsf T}$ and
$z=\dd_H^{-1/2}\one$. Then
$S_H=(r/n)zz^{\mathsf T}+\dd_H^{-1/2}E\dd_H^{-1/2}$, and the
second term has norm at most $\|E\|/\delta(H)=o(1)$.
The first term is positive semidefinite, so every unit vector has
Rayleigh quotient at least $-o(1)$ under $S_H$. On the subspace
orthogonal to $z$, the rank-one term vanishes, so every unit vector
in that subspace has Rayleigh quotient at most $o(1)$. The variational
characterization of eigenvalues therefore bounds the second-largest
eigenvalue from above by $o(1)$ and the smallest eigenvalue from below by $-o(1)$. Since $S_H$ has
eigenvalue one, this must be the largest eigenvalue of $S_H$ for all
sufficiently large $n$. The conclusion follows.
\end{proof}

\begin{lemma}\label{lem:random}
Suppose that $d=np\gg\log n$ and $p=o(1)$. Whp,
$G\sim G(n,p)$ satisfies $|D(G)|=(1/2+o(1))n$,
$d_G(v)=(1+o(1))d$ for every vertex, and
\[
\Delta_2(G)=O(d^2/n+\log n),\qquad
\left\|A_G-\frac dn\one\one^{\mathsf T}\right\|
=O\bigl(\sqrt{d\log n}\bigr).
\]
Consequently, $\Delta_2(G)=o(d)$ and $\lambda(G)=o(1)$.
\end{lemma}
\begin{proof}
The odd-degree assertion is Claim~2.2 of~\cite{KKS15}; the
hypotheses $np\to\infty$ and $n(1-p)\to\infty$ hold here.
Let $\mu=(n-1)p$ and $t=C\sqrt{d\log n}$, where $C\ge10$
is fixed. Each degree has distribution $\operatorname{Bin}(n-1,p)$,
and $t/\mu=o(1)$. The Chernoff bounds give, for every $v$,
\[
 \Pr\bigl(|d_G(v)-\mu|\ge t\bigr)
 \le2\exp\left(-\frac{C^2d\log n}{3(n-1)p}\right)
 \le2n^{-C^2/3}\le n^{-3}.
\]
Consequently, $
 \Pr\bigl(\exists v:|d_G(v)-\mu|\ge t\bigr)
 \le\sum_{v\in V(G)}\Pr\bigl(|d_G(v)-\mu|\ge t\bigr)
 \le n^{-2}=o(1).$
Since $|d-\mu|=p$, all degrees are
$d+O(\sqrt{d\log n})=(1+o(1))d$ whp.
For distinct $u,v$, let $Y_{uv}=|N_G(u)\cap N_G(v)|$.
The variable $Y_{uv}$ is binomial with mean
$\mu_{uv}=(n-2)p^2\le d^2/n$, because distinct possible common
neighbors use disjoint pairs of edges. With
$t_0=C(d^2/n+\log n)$, the Chernoff bounds imply $
 \Pr(Y_{uv}\ge t_0)
 \le(e\mu_{uv}/t_0)^{t_0}
 \le(e/C)^{t_0}\le n^{-C\log(C/e)}\le n^{-4}.$
It follows that
$\Pr(\Delta_2(G)\ge t_0)\le\sum_{u<v}\Pr(Y_{uv}\ge t_0)
\le n^{-2}=o(1)$.

For the matrix estimate, apply Tropp's matrix Bernstein
inequality~\cite[Theorem~1.4]{Tropp12} to
$A_G-\mathbb E A_G$. For each possible edge $uv$, the corresponding
independent symmetric summand has its only nonzero entries at
positions $(u,v)$ and $(v,u)$, both equal to the indicator of
$uv\in E(G)$ minus $p$.
Each summand has mean zero and norm at most one. The expected square
has diagonal entries $p(1-p)$ at $u,v$ and zero elsewhere.
The sum of these expected squares is $(n-1)p(1-p)I$, where $I$ is the
identity matrix, so the variance parameter is at most $d$.
Applying the inequality to the sum and its negative yields
\[
\Pr\bigl(\|A_G-\mathbb E A_G\|\ge t\bigr)
\le2n\exp\left(-\frac{t^2}{2d+2t/3}\right).
\]
At $t=C\sqrt{d\log n}$ we have $t=o(d)$, so the denominator is
at most $3d$ for large $n$. The failure probability is therefore
at most $2n^{1-C^2/3}=o(1)$ for $C\ge10$. Since
$\mathbb E A_G=p(\one\one^{\mathsf T}-I)$, we obtain
\[
 \left\|A_G-\frac dn\one\one^{\mathsf T}\right\|
 \le\|A_G-\mathbb E A_G\|+\|pI\|
 \le C\sqrt{d\log n}+p=O(\sqrt{d\log n}).
\]
The probability that any of the four asserted properties fails is
at most the sum of the four failure probabilities, hence is $o(1)$.
Finally,
$d/n=p=o(1)$ and $\log n=o(d)$ imply $\Delta_2(G)=o(d)$, while
Lemma~\ref{lem:normalize} proves $\lambda(G)=o(1)$.
\end{proof}

We use the following consequence of the constructions
in~\cite[proofs of Theorems~5.1 and~5.3]{KKS15}.
\begin{lemma}[Kor\'andi--Krivelevich--Sudakov~\cite{KKS15}]\label{lem:core}
Suppose that $\log^6 n/n\le q=o(1)$. Whp,
$G_1\sim G(n,q)$ has an edge partition $G_1=H\dot\cup H_0$ such
that $H$ is even, $\Delta(H_0)=O(\log^2 n)$, and $H$ has a
decomposition into $o(n)$ cycles.
\end{lemma}

\begin{proof}
	The maximum-degree conclusion follows from the constructions, rather than just the stated number of pieces. In both cited proofs,
	the parity-correction edges are denoted by $E_0$; we write $H_0$ for
	the graph on the same vertex set whose edge set is $E_0$.
	
	 In the proof of
	\cite[Theorem~5.1, p.~870]{KKS15}, $H_0$ lies in the layer denoted by $G_2$ in that proof,
	and removing $H_0$ leaves a graph with a decomposition into $o(n)$ cycles. In the proof of
	\cite[Theorem~5.3, pp.~870--871]{KKS15}, $H_0$ lies in the layer denoted by $G_3$ in that proof,
	and the remaining graph has a decomposition into at most
	$nq/2+o(n)=o(n)$ cycles.
	
	The two ranges are respectively $q\le n^{-1/6}$ and $q\ge n^{-1/6}$.
	In both cases, removing $H_0$ yields an even graph $H = G_1 - H_0$.
	Both cited layers have distribution $G(n,\log^2 n/n)$.
	The Chernoff bounds, summed over the vertices, show that each
	layer has maximum degree at most $2\log^2 n$ except with probability
	$n e^{-\Omega(\log^2 n)}=o(1)$.
	Hence $\Delta(H_0) = O(\log^2 n)$ holds whp.
\end{proof}

\section{Pairing incident edges}\label{sec:pairing}

An \emph{incidence} at a vertex $v$ is a pair $(v,e)$ with $v\in e\in E(H)$.
A \emph{pair} refers to a two-element subset of a set.
A \emph{pairing} of a finite set of even cardinality is a partition of the set into disjoint pairs.
We will choose pairings of incidences independently at each vertex.

We use the following form of the Lov\'asz local lemma. In the
statement, a graph on a collection of events specifies which events
are neighbors; $\Gamma(\mathcal E)$ denotes the set of neighbors of
$\mathcal E$. The graph is not required to describe independence.
The conditional inequality below is the property we need.
\begin{lemma}[lopsided Lovász local lemma, {\cite[Lemma~1]{LSmatch}}]\label{lem:lopsided}
Let $\mathcal B$ be a finite collection of events with a graph on
$\mathcal B$. Suppose that, for each $\mathcal E\in\mathcal B$
and every $\mathcal F\subseteq
\mathcal B\setminus(\Gamma(\mathcal E)\cup\{\mathcal E\})$,
\[
 \Pr\left(\mathcal E\,\middle|\,
       \bigcap_{\mathcal A\in\mathcal F}\overline{\mathcal A}\right)
 \le\Pr(\mathcal E)
\]
whenever $\Pr\big(\bigcap_{\mathcal A\in\mathcal F}\overline{\mathcal A}\big)>0$. If there exist numbers $x_{\mathcal E}\in[0,1)$ such that
$\Pr(\mathcal E)\le x_{\mathcal E}\prod_{\mathcal A\in\Gamma(\mathcal E)}(1-x_{\mathcal A})$
for all $\mathcal E\in\mathcal B$,
then $\Pr\bigl(\bigcap_{\mathcal E\in\mathcal B}\overline{\mathcal E}\bigr)>0$.
\end{lemma}

The proof of the following key lemma relies on random pairings of incidences, which induce a decomposition into trails. We use Lemma \ref{lem:lopsided} to ensure that these trails are simple paths, with the required bounds obtained through a switching argument and spectral estimates for nonbacktracking walks.

\begin{lemma}[Key Lemma]\label{lem:pairing}
Let $H$ be a simple graph on $n$ vertices. Suppose that a positive
sequence $r=r(n)$ and a fixed constant $c>0$ satisfy
$\log n=o(r)$, $r^3=o(n)$ and $|D(H)|\ge cn$. Suppose also that
$d_H(v)=(1+o(1))r$ uniformly over $v$, $\Delta_2(H)=o(r)$ and
$\lambda(H)=o(1)$. Then $H$ has a decomposition into exactly $|D(H)|/2$ paths for all sufficiently large $n$.
\end{lemma}
\begin{proof}
We may assume that $c\le1$ and that $n$ is large enough for
$3r/4\le d_H(v)\le5r/4$ and $\lambda(H)\le1/2$.
At each even-degree vertex, choose a pairing of the incidences
uniformly from all such pairings. At each odd-degree vertex, add
one auxiliary object, to the set of
incidences and choose a uniform pairing of the enlarged set; we
refer to the original incidences as \emph{real incidences} to
distinguish them from the auxiliary object. When tracing trails, ignore the auxiliary object and treat its real partner as unpaired. The choices at different vertices are independent. Thus exactly one real incidence is unpaired at every vertex of $D(H)$,
and none is unpaired elsewhere.

To describe the resulting trails, form a graph with vertex set
$E(H)$, joining two vertices when the corresponding
edges are paired at a common endpoint. Each vertex of this graph has degree at most two, as an edge can belong to at most two pairs, one at each of its endpoints. Thus every component of this auxiliary graph is an isolated vertex, a path, or a cycle. No edge of $H$ is
repeated, since it occurs as only one vertex of the auxiliary graph. Traversing a component corresponds to following a sequence of edges paired, which forms a trail in $H$. A path component, including an isolated vertex, corresponds in $H$ to a trail with two unpaired ends; a cycle component corresponds to a closed trail. Each component of the first kind uses two unpaired incidences, and every unpaired incidence occurs
in one such component. There are exactly $|D(H)|/2$ components of the
first kind. It remains
to exclude closed trails and repeated vertices in the other trails.

For $v\in D(H)$ let $q_v=d_H(v)$, and for $v\notin D(H)$ let
$q_v=d_H(v)-1$. For large $n$, these numbers satisfy
$r/2\le q_v\le2r$. For each cycle $C$ and each vertex $z\in V(C)$, called the
\emph{marked vertex}, define a \emph{bad event} $\mathcal E_{C,z}$
that, at every vertex $v\in V(C)\setminus\{z\}$, the two incidences
corresponding to the two edges of $C$ incident with $v$ are paired
with each other. A prescribed pair occurs in a uniform pairing of $m$ objects with probability $1/(m-1)$. Independence between vertices therefore implies
\begin{equation}\label{eq:event}
\Pr(\mathcal E_{C,z})
=\prod_{v\in V(C)\setminus\{z\}}q_v^{-1}.
\end{equation}
If a trail repeats a vertex, choose a shortest nonempty segment
whose first and last vertices coincide. The other vertices of the
segment are distinct, and the length is at least three because $H$ is simple
and the trail repeats no edge. The segment is therefore a cycle $C$. Taking the repeated vertex as
the mark $z$, at every vertex of $C$ other than $z$ the pairing contains the pair prescribed by $\mathcal E_{C,z}$. The same
argument applies to a closed trail. Hence, if all these events are avoided, there is neither a closed trail nor a trail that repeats a vertex, so every component of the auxiliary graph is a path component and the corresponding trails in $H$ are simple paths. There are exactly $|D(H)|/2$ such components, which gives the required decomposition and proves the lemma. We now show that all these events can be avoided simultaneously.

The \emph{support} of $\mathcal E_{C,z}$ is the set of real incidences in the prescribed pairs. Join two distinct bad events
when the supports intersect, and use
$\mathcal E\sim\mathcal A$ when the events are joined.

\begin{claim}\label{clm:dependence}
For every bad event $\mathcal E$ and every collection $\mathcal F$
of nonneighbors of $\mathcal E$ not containing $\mathcal E$,
\[
 \Pr\left(\mathcal E\,\middle|\,
       \bigcap_{\mathcal A\in\mathcal F}\overline{\mathcal A}\right)
 \le\Pr(\mathcal E)
\]
whenever the conditioning event has positive probability.
\end{claim}
\begin{proof}[Proof of Claim~\ref{clm:dependence}]
We establish the claim by a standard switching argument on the pairings.
Fix two distinct objects $a,b$ in an even set $I$ of size $m$,
and let $\mathcal M(I)$ be the set of pairings of $I$. Define
$\phi_{ab}:\mathcal M(I)\to\{M\in\mathcal M(I):ab\in M\}$ by
\[
 \phi_{ab}(M)=
 \begin{cases}
 M,&ab\in M,\\
 (M\setminus\{aa',bb'\})\cup\{ab,a'b'\},
   &aa',bb'\in M\text{ and }ab\notin M.
 \end{cases}
\]
Here $ab$ denotes the unordered pair $\{a,b\}$. Each pairing
containing $ab$ has exactly $1+2(m/2-1)=m-1$ preimages: the pairing
itself, and two preimages for each other pair $a'b'$, obtained by
replacing $ab,a'b'$ with $aa',bb'$ or $ab',ba'$. Therefore
$\phi_{ab}$ sends a uniform pairing to a uniform pairing conditional
on containing $ab$.

A bad event $\mathcal E$ prescribes at most one pair at each
vertex. Apply this map at each vertex where $\mathcal E$ prescribes a pair,
using the enlarged set $I$ (including the auxiliary object) at an
odd vertex, and leave all other pairings unchanged; denote the
resulting global map by $\Phi_{\mathcal E}$. Since the
original pairings at distinct vertices are independent and every
map acts only at one vertex, $\Phi_{\mathcal E}$ gives exactly
the distribution of all pairings conditional on $\mathcal E$.

The map $\phi_{ab}$ removes only pairs containing $a$ or $b$.
Consequently, $\Phi_{\mathcal E}$ preserves every pair disjoint
from the support of $\mathcal E$. If an event in $\mathcal F$
holds before applying $\Phi_{\mathcal E}$, all prescribed pairs
for that event remain present afterwards. Thus avoiding every
event of $\mathcal F$ after applying $\Phi_{\mathcal E}$ implies
avoiding every event of $\mathcal F$ beforehand. The original and
conditional distributions therefore satisfy
\[
 \Pr\left(\bigcap_{\mathcal A\in\mathcal F}\overline{\mathcal A}
      \,\middle|\,\mathcal E\right)
 \le\Pr\left(\bigcap_{\mathcal A\in\mathcal F}
                        \overline{\mathcal A}\right).
\]
Bayes' rule proves the claim. This argument does not require events
using disjoint incidences at the same vertex to be independent.
\end{proof}

We now estimate the total weight of events using one incidence.
Let $P=P_H$, $S=S_H$, $\lambda=\lambda(H)$ and
$\vol(H)=\sum_{v\in V(H)}d_H(v)$. For
$w=(\sqrt{d_H(v)}:v\in V(H))$, the matrix
$\Pi=ww^{\mathsf T}/\vol(H)$ satisfies
$S\Pi=\Pi S=\Pi$ and $\Pi^2=\Pi$. Since $\Pi$ projects onto the eigenspace of the eigenvalue $1$, we
have $\|S-\Pi\|=\lambda$.
If $\mathbf e_v$ has entry one at $v$ and zero elsewhere, then
\[
\|(S-\Pi)\mathbf e_v\|_2^2
=(S^2-\Pi)_{vv}\leq (S^2)_{vv}
=\sum_{u\in N_H(v)}\frac{1}{d_H(v)d_H(u)}
\le\frac2r.
\]
For $k\ge2$, the identity
$S^k-\Pi=(S-\Pi)(S-\Pi)^{k-2}(S-\Pi)$ and Cauchy--Schwarz give $
 |(S^k-\Pi)_{vu}|
 \le\|(S-\Pi)\mathbf e_v\|_2\,
       \|(S-\Pi)^{k-2}\|\,
       \|(S-\Pi)\mathbf e_u\|_2
 \le\frac2r\lambda^{k-2}.$
Since $3r/4\le d_H(v)\le5r/4$, we have $d_H(u)/\vol(H)$ is at most $2/n$. Since $P^k=\dd_H^{-1/2}S^k\dd_H^{1/2}$ and
$d_H(u)/d_H(v)\le5/3<4$, for $k\geq 2$, we have
\begin{equation}\label{eq:walk}
	\begin{aligned}
		\left|(P^k)_{vu}-\frac{d_H(u)}{\vol(H)}\right|
		&=\sqrt{\frac{d_H(u)}{d_H(v)}}\,|(S^k-\Pi)_{vu}|
		\le\frac4r\lambda^{k-2},\\
		(P^k)_{vu}
		&\le\frac{d_H(u)}{\vol(H)}+\frac4r\lambda^{k-2}
		\le\frac2n+\frac4r\lambda^{k-2}.
	\end{aligned}
\end{equation}

Replace each undirected edge $uv$ of $H$ by the two directed
edges $(u,v)$ and $(v,u)$. Let $B$ be the square matrix of order
$2|E(H)|$ indexed by these directed edges, with
$B_{(u,v),(v,y)}=q_v^{-1}$ for $y\ne u$ and all other entries zero.
A \emph{nonbacktracking walk} never follows an edge
immediately by the same edge in the opposite direction.
Thus $(B^s)_{e,f}$ is the sum of the products of the transition
weights over nonbacktracking walks from directed edge $e$ to
directed edge $f$ in $s$ transitions. These are counting weights;
they do not assume independent choices along a pairing trail.
For a row indexed by $(u,v)$, its entries sum to $1$ when $v$ is an
even-degree vertex, and to $1-1/d_H(v)$ when $v\in D(H)$. In
particular, $\|B\|_\infty\le1$.

Let $b=c/16$ and $T=\lceil8\log n\rceil$. We state and prove the following claim.
\begin{claim}\label{clm:decay}
	$\|B^s\|_\infty\le2e^{-bs/r}$ for all integers $s\ge 0$.
\end{claim}
\begin{proof}[Proof of Claim~\ref{clm:decay}]
The stationary probability of a vertex $u$ for the random walk is
$d_H(u)/\vol(H)$. Since $d_H(v)=(1+o(1))r$ uniformly over $v$, we have $\delta(H)/\Delta(H)\to1$. The total stationary probability of the odd-degree vertices
satisfies $\sum_{u\in D(H)}\frac{d_H(u)}{\vol(H)}\ge\frac{|D(H)|\delta(H)}{n\Delta(H)}\ge c\,\frac{\delta(H)}{\Delta(H)}\ge\frac{3c}{4}$ for large $n$.
 Summing~\eqref{eq:walk} over $D(H)$, with $\lambda\le1/2$ and
$s\ge T$, gives $\sum_{u\in D(H)}(P^s)_{vu}\ge\frac{3c}{4}-\frac{4n}{r}\lambda^{T-2}\ge\frac{3c}{4}-\frac{16}{r}n^{1-8\log 2}\ge\frac c2$.

Let $\chi$ be the vector on directed edges with $\chi_{(u,v)}=1$ if
$v\in D(H)$, and zero otherwise. Consider a random walk starting at
$v$, with $u$ regarded as the preceding vertex. At each step, the
conditional probability of immediately returning to the preceding
vertex is at most $1/\delta(H)$, including a first step back to $u$.
Summing over the first $s$ steps bounds the probability of at least
one such return by $s/\delta(H)$. For a walk
$v=x_0,x_1,\ldots,x_s$ with no such return, the product of the
$B$-entries is $
 \prod_{i=0}^{s-1}q_{x_i}^{-1}
 \ge\prod_{i=0}^{s-1}d_H(x_i)^{-1},$
where the right-hand side is the probability of the same walk under
the random-walk matrix $P$.
For $T\le s<2T$, it follows that
\[
(B^s \chi)_{(u,v)}
\ge\sum_{x\in D(H)}(P^s)_{vx}-\frac{s}{\delta(H)}
\ge\frac c4,
\]
since $s/\delta(H)\le2T/\delta(H)=o(1)\le c/4$ for large $n$.
For every directed edge $(u,v)$, $
 (\one-B\one)_{(u,v)}=
 \begin{cases}
  1/d_H(v),&v\in D(H),\\
  0,&v\notin D(H).
 \end{cases}$ Hence $\one-B\one\ge {\chi}/{2r}$.
All vector inequalities are entrywise. Since $B$ has nonnegative
entries, telescoping gives
\[
 \one-B^{2T}\one
 =\sum_{s=0}^{2T-1}B^s(\one-B\one)
 \ge\frac1{2r}\sum_{s=T}^{2T-1}B^s\chi
 \ge\frac{cT}{8r}\one.
\]
Since $B^{2T}$ is nonnegative,
$\|B^{2T}\|_\infty\le1-cT/(8r)$. Let $s=2Tj+h$, where
$j=\lfloor s/(2T)\rfloor$ and $0\le h<2T$. We have
$0<cT/(8r)<1$, $j\ge s/(2T)-1$ and $T/r\to0$.
Using $(1-x)^j\le e^{-xj}$, the product inequality for the
row-sum norm, and $\|B^h\|_\infty\le1$, we obtain $
 \|B^s\|_\infty
 \le\left(1-\frac{cT}{8r}\right)^j
 \le e^{-cTj/(8r)}
 \le e^{cT/(8r)}e^{-cs/(16r)}
 \le2e^{-cs/(16r)}.$
This proves Claim~\ref{clm:decay}.
\end{proof}

Choose $a=b/2$ and assign each bad event $\mathcal E_{C,z}$ the weight
$x_{C,z}=e^{a|C|/r}\Pr(\mathcal E_{C,z})$. For a real incidence
$h$, let $L_h$ be the sum of these weights over events whose
support contains $h$.
\begin{claim}\label{clm:load}
For every real incidence $h$, $
L_h=O\left(\frac{\Delta_2(H)}{r^2}+\frac{\lambda}{r}
+\frac{r^2}{n}\right)=o(1/r)$ 
uniformly over $h$.
\end{claim}
\begin{proof}[Proof of Claim~\ref{clm:load}]
Suppose that $h$ is the incidence of $uv$ at $v$, and let
$e=(u,v)$. Every cycle of length $\ell$ containing $uv$ has a
unique direction starting with $e$. The resulting closed walk contributes $\prod_{x\in V(C)}q_x^{-1}$
to $(B^\ell)_{e,e}$, 
so $\prod_{x\in V(C)}q_x^{-1}\le (B^\ell)_{e,e}$. For a
mark $z$, equation~\eqref{eq:event} gives $
\Pr(\mathcal E_{C,z})
=q_z\prod_{x\in V(C)}q_x^{-1}
\le 2r\prod_{x\in V(C)}q_x^{-1}.$ Since there are at most $\ell$ marks, the total over the marks is
at most $2r\ell\prod_{x\in V(C)}q_x^{-1}$. 
Consequently, since $x_{C,z}=e^{a\ell /r}\Pr(\mathcal E_{C,z})$, we have

\begin{equation}\label{eq:return}
L_h\le2r\sum_{\ell=3}^{n}\ell e^{a\ell/r}(B^\ell)_{e,e}.
\end{equation}
For $\ell=3$, the edge $uv$ belongs to at most $\Delta_2(H)$
triangles. Each has at most three marks, and each corresponding
probability in~\eqref{eq:event} is $O(r^{-2})$. The contribution
from triangles is therefore $O(\Delta_2(H)/r^2)$.

For the remaining terms, we compare $B$ with $P$. We have $d_H(x)/q_x=1$ at odd-degree vertices and
$d_H(x)/q_x=1+1/q_x\le e^{2/r}$ at even-degree vertices since $1/q_x\le2/r$, so
$1/q_x\le e^{2/r}/d_H(x)$.	
A contributing walk makes $m-1$ steps from $v$ to $x$ and then
the prescribed step $xy$. Comparing each of those first steps with
its $P$-probability, for $m\geq 1$, we have
\begin{equation}\label{eq:comparison}
(B^m)_{(u,v),(x,y)}
\le\frac2r e^{2(m-1)/r}(P^{m-1})_{vx}\leq \frac2r e^{2m/r}(P^{m-1})_{vx}.
\end{equation}
For $m=1$, the formula uses $P^0=I$ and remains valid.
Since $T=o(r)$, $e^{2m/r}\le e^{2T/r}=e^{o(1)}=O(1)$ for $m\le T$. The identity
$\sum_{j\ge1}(j+3)\lambda^j
=\lambda/(1-\lambda)^2+3\lambda/(1-\lambda)$
shows that the series is $O(\lambda)$, since $\lambda\le1/2$. Thus~\eqref{eq:walk} and~\eqref{eq:comparison}
bound the terms $4\le\ell<T$ in~\eqref{eq:return} by
\[
O\left(\sum_{\ell=4}^{T-1}\ell
\left(\frac1n+\frac{\lambda^{\ell-3}}r\right)\right)
=O\left(\frac{T^2}{n}+\frac\lambda r\right).
\]

At $m=T$, the same estimates imply
$(B^T)_{f,e}=O(\frac 2r (P^{m-1})_{vx})=O(1/(nr))$ for all directed edges $f,e$, since
$\lambda^{T-3}=O(n^{-8\log2})=O(n^{-5})$. Then $(B^T)_{f,e}\le C/(nr)$ for some absolute constant $C>0$.
For $\ell\ge T$, Claim~\ref{clm:decay} and $e^{bT/r}=O(1)$ yield
\[
(B^\ell)_{e,e}
=\sum_f(B^{\ell-T})_{e,f}(B^T)_{f,e}
\le\frac{C}{nr}\|B^{\ell-T}\|_\infty
=O\left(\frac{e^{-b\ell/r}}{nr}\right).
\] All these bounds
are uniform over the directed edges. Consequently, the contribution
of these terms to~\eqref{eq:return}
is $O(n^{-1}\sum_{\ell\ge T}\ell e^{-(b-a)\ell/r})=O(r^2/n)$.
Indeed, with $z=e^{-(b-a)/r}$,
$\sum_{\ell\ge1}\ell e^{-(b-a)\ell/r}=z/(1-z)^2=O(r^2)$,
since $r(1-z)\to b-a>0$. Since
$T=o(r)$, we also have
$T^2/n=O(r^2/n)$. Combining the three ranges proves the first
bound in Claim~\ref{clm:load}. Multiplication by $r$ leaves
$O(\Delta_2(H)/r+\lambda+r^3/n)=o(1)$, proving the second bound.
\end{proof}

We finish by checking the numerical condition of Lemma~\ref{lem:lopsided}.
Let $L=\max_h L_h$. Since the support of an event on a cycle of length $\ell$ consists
of $2(\ell-1)$ incidences, the sum of the weights of its
neighbors is at most $2(\ell-1)L$, as every neighbor contains at
least one of these incidences. By~\eqref{eq:event}, we have $x_{C,z}=e^{a\ell/r}\Pr(\mathcal E_{C,z})\le b_\ell:=e^{a\ell/r}(2/r)^{\ell-1}$. Consecutive bounds satisfy $b_{\ell+1}/b_\ell=2e^{a/r}/r<1$ for large $n$. Then $b_\ell$ decreases with $\ell$ and is largest at $\ell=3$, and thus $x_{C,z}\leq b_\ell\leq  b_3=e^{3a/r}(2/r)^2=O(r^{-2})<1/2$. By
Claim~\ref{clm:load}, $L=o(1/r)$. Since $a>0$ is fixed, this
implies $4L\le a/r$ for large $n$. Using $1-x\ge e^{-2x}$ for
$0\le x\le1/2$, we obtain
\[
x_{C,z}\prod_{\mathcal E_{C',z'}\sim\mathcal E_{C,z}}(1-x_{C',z'})
\ge e^{a\ell/r}\Pr(\mathcal E_{C,z})e^{-4(\ell-1)L}
\ge\Pr(\mathcal E_{C,z}).
\]
 Since the displayed inequality is the numerical condition of
 Lemma~\ref{lem:lopsided} and Claim~\ref{clm:dependence} verifies
 its conditional probability condition, Lemma~\ref{lem:lopsided}
 gives a positive probability of avoiding all bad events. The
 resulting trails are exactly $|D(H)|/2$ simple paths.
\end{proof}

\section{Proof of the main theorem}\label{sec:main}
For each cycle $C$, we will select two distinct vertices of $V(C)$
at which additional edges may be attached. We call them its
\emph{attachment vertices}.

We use the classical bound that a graph on $k$ vertices can be
decomposed into $O(k\log k)$ cycles and single edges; see
\cite[Section~1]{CFS14}. For an even graph, the single edges left
after removing those cycles also form an even graph. Decomposing
that remainder into cycles uses no more cycles than the number
of its edges. It follows that there is an absolute constant $A$
such that every even graph with $k\le n$ nonisolated vertices has
a decomposition into at most $Ak\log n$ cycles, for all sufficiently
large $n$. 

We use the block decomposition of a graph and the fact that any
two vertices of a 2-connected graph lie on a common cycle;
see~\cite[Theorem~5.1 and Proposition~5.3]{Bondy2008}. A \emph{cut vertex} is a vertex whose removal increases the component count. A \emph {2-connected} graph is connected on at least three vertices with no cut vertex. A \emph{bridge} is an edge whose removal increases the component count. A \emph{block} is a maximal subgraph with no cut
vertex, including a bridge.

A \emph{capacity} bounds the flow on an arc, and the
\emph{flow} entering each intermediate vertex must equal the flow leaving
it. A \emph{cut} is the set of arcs separating the source from the sink; its capacity is the
sum of capacities of arcs directed from the source side to the
sink side. An arc is \emph{saturated} if its flow equals its capacity.

\begin{lemma}\label{lem:attachments}
Let $H$ be an even graph on $n$ vertices, and let $\mathcal C$ be
a cycle decomposition using the fewest cycles. Each cycle
$C\in\mathcal C$ can be assigned two distinct attachment vertices
$x_C,y_C\in V(C)$ so that every vertex is assigned at most
$K=\lceil4A\log n\rceil$ times, for all sufficiently large $n$.
\end{lemma}
\begin{proof}
Fix $U\subseteq V(H)$. Let $q_0$ count the cycles of $\mathcal C$
contained in $U$, and $q_1$ those with exactly one vertex outside
$U$. Let $H_1$ be the union of these cycles. The $q_0+q_1$ chosen cycles
form a minimum cycle decomposition of $H_1$, since a smaller one could replace
them in $\mathcal C$, leaving every cycle outside $H_1$ unchanged.
When $H_1$ has no edges, $q_0=q_1=0$, so the estimate
$2q_0+q_1\le K|U|$ established below holds trivially.

As long as $H_1$ contains a cycle through two distinct vertices
outside $U$, delete such a cycle. The initial sum of the degrees
outside $U$ is $2q_1$, and each deletion reduces this sum by at
least four. Thus at most $q_1/2$ cycles are deleted. The remaining
graph $H_1'$ is even and contains no cycle through two vertices
outside $U$.

Since $H_1'$ is even, every edge of $H_1'$ lies on a cycle, so $H_1'$ has no bridges. Hence every block of $H_1'$ is 2-connected and has at least three vertices. A block containing two vertices outside $U$ has a cycle through
both vertices by its 2-connectivity, contradicting the choice of
$H_1'$. Thus each block has at most one vertex outside $U$,
and hence at least two vertices in $U$. Order the blocks in each component of $H_1'$ so that each new block meets the preceding ones in one cut vertex. The first block has at least two vertices in $U$ and at most one
outside, so it contains no fewer vertices in $U$ than outside.
Every later block shares a single cut vertex with the preceding
union, and all its other vertices are new. If that cut vertex lies
in $U$, the block adds at least one new vertex in $U$ and at most
one outside; if it lies outside $U$, all its new vertices lie in
$U$. Thus each block adds no fewer vertices in $U$ than outside.
 Summing over the components, which have disjoint
vertex sets, shows that $H_1'$ has at most $2|U|$ nonisolated vertices.

Since $n$ is sufficiently large, the preceding cycle bound
decomposes $H_1'$ into at most
$2A|U|\log n$ cycles. Combining these with the deleted cycles,
and using minimality for $H_1$, we obtain
$q_0+q_1\le q_1/2+2A|U|\log n$. Hence
$2q_0+q_1\le4A|U|\log n\le K|U|$.

To make the assignments, apply the integral max-flow min-cut
theorem~\cite[Section~7.2 and Exercise~7.2.2]{Bondy2008}. Construct a directed
network with a source, one vertex for each cycle, one vertex for
each vertex of $H$, and a sink. The arcs from the source to the
cycles have capacity two, the arcs from each cycle to its vertices
have capacity one, and the arcs from graph vertices to the sink
have capacity $K$. If $\mathcal S$ and $U$ are the cycles and graph
vertices on the source side of a cut, its capacity is
\[
2(|\mathcal C|-|\mathcal S|)
+\sum_{C\in\mathcal S}|V(C)\setminus U|+K|U|.
\]
Only cycles with zero or one vertex outside $U$ contribute
positively, and they contribute at most $2q_0+q_1$ in total; hence
$\sum_{C\in\mathcal S}(2-|V(C)\setminus U|)\le2q_0+q_1\le K|U|$.
Subtracting $2|\mathcal C|$ from the cut capacity leaves
$K|U|-\sum_{C\in\mathcal S}(2-|V(C)\setminus U|)\ge0$,
so every cut has capacity at least $2|\mathcal C|$. The total capacity leaving
the source is exactly this value. The integral max-flow min-cut theorem therefore
provides a flow of value $2|\mathcal C|$ that is integer on every
arc. Every source-to-cycle arc is saturated, since the capacities of
these arcs sum to the flow value. At each cycle its two units of flow go to distinct vertices because those arcs
have capacity one. The capacity $K$ at each graph vertex bounds
its number of assignments, as required.
\end{proof}

A \emph{chord} of a cycle $C$ is an edge joining two non-consecutive vertices of $C$.

\begin{lemma}\label{lem:two-paths}
	Each of the following graphs admits a decomposition into two paths whose four endpoints are exactly $a,b,c,d$, each occurring once:
	\smallskip
	
	{\rm (i)} a cycle $C$ together with edges $xa,xb,yc,yd$, where $x\neq y$ lie on $C$ and $a,b,c,d$ are distinct vertices outside $C$;
	
	{\rm (ii)} a cycle $C$ together with two vertex-disjoint chords $ab,cd$.
\end{lemma}

\begin{proof}
	For {\rm (i)}, split $C$ into two edge-disjoint $x$--$y$ paths. Concatenate the edge $xa$, one of the two $x$--$y$ paths, and the
	edge $yc$ to form a path with endpoints $a,c$; concatenate the edge
	$xb$, the other $x$--$y$ path, and $yd$ to form a path with endpoints $b,d$. For {\rm (ii)}, the four chord endpoints divide \(C\) into four edge-disjoint subpaths. Concatenating two suitably chosen subpaths with the chords yields two paths whose endpoints are exactly $a,b,c,d$.
\end{proof}

\begin{proof}[Proof of Theorem~\ref{thm:main}]
Let $d=np$ and $D=D(G)$. First suppose that $d\le n^{1/4}$.
Lemma~\ref{lem:random} supplies the degree, common-neighbor $\Delta_2(G)$,
spectral and odd-degree hypotheses of Lemma~\ref{lem:pairing}
with $r=d$ and $c=1/4$, whp. Its other hypotheses hold because
$\log n=o(d)$ and $d^3\le n^{3/4}=o(n)$. Thus
Lemma~\ref{lem:pairing} proves the theorem in this range.

We may therefore assume that $d>n^{1/4}$. Let $r=n^{1/8}$.
We partition the edges of $G$ into two parts $G_1,G_2$ that are
edge-disjoint but not independent. Independently for each unordered
pair of distinct vertices, assign the labels $G_1$, $G_2$, or absent
with probabilities $p-r/n$, $r/n$, and $1-p$, respectively.
These probabilities are nonnegative for large $n$. Consequently,
$G=G_1\dot\cup G_2\sim G(n,p)$,
$G_1\sim G(n,q)$ for $q=p-r/n$, and $G_2\sim G(n,r/n)$.

Expose only $G_1$. Since
$q\ge\tfrac12n^{-3/4}\gg\log^6 n/n$ and $q=o(1)$,
Lemma~\ref{lem:core} supplies, whp, a partition
$G_1=H\dot\cup H_0$ with $H$ even, $\Delta(H_0)=O(\log^2 n)$,
and a decomposition of $H$ into $o(n)$ cycles.
Also, $\Delta(G_1)=o(n)$ whp by Lemma~\ref{lem:random}.
Choose such a partition and a minimum cycle decomposition
$\mathcal C$ of $H$, breaking ties by a deterministic rule
depending only on $G_1$. Let $t=|\mathcal C|=o(n)$.
Choose attachment vertices $x_C,y_C$ by Lemma~\ref{lem:attachments},
again deterministically from $G_1$, and let $A_*$ be their union. Thus $|A_*|\le2t$,
and every vertex of $A_*$ is selected as an attachment vertex
for at most $K=O(\log n)$ cycles. Call a cycle short if its length
is at most $3n/4$, and long otherwise.

\begin{claim}\label{clm:candidates}
For each fixed exposed graph $G_1$ as above with
$\Delta(G_1)\le n/100$ and $t\le n/100$, the conditional probability
that every short cycle $C\in\mathcal C$ and each of its two
attachment vertices $x$ satisfy
\begin{equation}\label{eq:candidates}
 |N_{G_2}(x)\cap(D\setminus(V(C)\cup A_*))|\ge r/100
\end{equation}
is at least $1-10n e^{-r/200}$ for all sufficiently large $n$.
\end{claim}
\begin{proof}[Proof of Claim~\ref{clm:candidates}]
Throughout this proof, condition on the fixed graph $G_1$.
Call an unordered pair of distinct vertices \emph{allowed} if
it is not an edge of $G_1$. Each allowed pair belongs to $G_2$
independently with probability
$\rho=(r/n)/(1-p+r/n)=(1+o(1))r/n$ since $E(G_1)
\cap E(G_2)=\emptyset$.

Fix a short cycle $C$ and an attachment vertex $x$ on it.
There are at least $n/4-2t-\Delta(G_1)\ge n/5$ allowed pairs
from $x$ to $V(G)\setminus(V(C)\cup A_*)$.
Thus $Y=|N_{G_2}(x)\setminus(V(C)\cup A_*)|$ is binomial with
conditional mean $\mu\ge r/5$, since $\rho\ge r/n$.
The relative Chernoff bound with deviation $1/2$ gives $\Pr(Y<r/10\mid G_1)
 \le \Pr(Y<\mu/2\mid G_1)
 \le e^{-\mu/8}\le e^{-r/40}.$
The conditional mean $\mu_x$ of $d_{G_2}(x)$ satisfies
$r/2\le(n-1-\Delta(G_1))\rho\le\mu_x\le(n-1)\rho\le4r/3$
for large $n$. Hence $2r\ge3\mu_x/2$, and the relative
Chernoff bound with deviation $1/2$ gives $
 \Pr(d_{G_2}(x)>2r\mid G_1)
 \le\Pr(d_{G_2}(x)>3\mu_x/2\mid G_1)
 \le2e^{-\mu_x/12}\le2e^{-r/24}.$

Expose the indicators of all allowed pairs incident with $x$,
determining $N=N_{G_2}(x)$. Next expose the indicators of all
allowed pairs with both ends in $N$. Let $\mathcal F$ be the information generated by $G_1$ and these
two exposures.
The other allowed pairs remain independent with probability $\rho$.
Restrict for now to exposures satisfying $Y\ge r/10$ and $|N|\le2r$.
For each $u\in N\setminus(V(C)\cup A_*)$, the unknown part of
its degree consists of allowed edges of $G_2$ from $u$ to
$V(G)\setminus(N\cup\{x\})$. Their number of possible pairs is
$m_u\ge n-|N|-1-\Delta(G_1)\ge n/2$.
For distinct vertices $u,u'\in N$, these collections are disjoint, since the pair $uu'$ was already exposed. Hence, conditional on $\mathcal F$,
the remaining degree contributions are independent variables
with distributions $\operatorname{Bin}(m_u,\rho)$.

The exposed part of each degree has a fixed parity, so the
binomial parity identity gives $|\Pr(u\in D\mid\mathcal F)-1/2|=(1-2\rho)^{m_u}/2\le e^{-2\rho m_u}/2\le e^{-r}/2$, since $\rho<1/2$ for large $n$. In particular, these conditional
probabilities are at least $1/3$, and the indicators of membership
in $D$ are conditionally independent.
Since $|N_{G_2}(x)\setminus(V(C)\cup A_*)|=Y\ge r/10$,
the variable $Q=|N_{G_2}(x)\cap(D\setminus(V(C)\cup A_*))|$
has conditional mean $\nu\ge r/30$. The relative Chernoff bound
with deviation $7/10$ therefore gives
\[
 \Pr(Q<r/100\mid\mathcal F)
 \le\Pr(Q<3\nu/10\mid\mathcal F)
 \le2e^{-49\nu/300}\le2e^{-r/200}.
\]
Let $B$ be the good exposure $\{Y\ge r/10,\ |N|\le2r\}$. By the
law of total probability,
$\Pr(Q<r/100\mid G_1)=\Pr(Q<r/100\mid B,G_1)\Pr(B\mid G_1)+\Pr(Q<r/100\mid\bar B,G_1)\Pr(\bar B\mid G_1)\le2e^{-r/200}+3e^{-r/200}=5e^{-r/200}$, where in the second term $\Pr(Q<r/100\mid\bar B,G_1)\le1$.
There are at most $2t\le2n$ pairs $(C,x)$, where $C$ is short
and $x\in\{x_C,y_C\}$. The union bound now gives $
 \Pr\!\left(\exists\, (C,x):
 |N_{G_2}(x)\cap(D\setminus(V(C)\cup A_*))|<r/100
 \,\middle|\,G_1\right)
 \le10n e^{-r/200}=o(1).$
The constants do not depend on the exposed graph, and the local
exposures are used separately for each pair $(C,x)$, so the union
bound needs no independence between their failure events. Note also
that $D$ was never conditioned on.
\end{proof}

The conditions on $G_1$ in Claim~\ref{clm:candidates} hold whp,
so averaging its conditional estimate proves that
\eqref{eq:candidates} holds simultaneously for all the stated
pairs $(C,x)$ whp. Applying Lemma~\ref{lem:random} to $G$ and $G_2$
also gives, whp, $|D|=(1/2+o(1))n$,
$d_{G_2}(v)=(1+o(1))r$ uniformly over $v$,
$\Delta_2(G_2)=O(\log n)$, and
\begin{equation}\label{eq:reserve-spectrum}
 \beta:=\left\|A_{G_2}-\frac rn\one\one^{\mathsf T}\right\|
 =O\bigl(\sqrt{r\log n}\bigr).
\end{equation}
All these properties hold simultaneously whp
by the union bound. Fix an outcome on which they hold. The rest
of the proof is deterministic.

For each attachment vertex of a short cycle, we need two different
endpoints outside that cycle. Define a bipartite graph with right part
$R=D\setminus A_*$ and left part
\[
 L=\{(C,x,i): C\in\mathcal C\text{ is short},\quad
                 x\in\{x_C,y_C\},\ i\in\{1,2\}\}.
\]
A vertex $(C,x,i)\in L$ represents one endpoint to be chosen at $x$
for the cycle $C$. Join it to $u\in R$ precisely when
$u\in N_{G_2}(x)\setminus V(C)$. There are at most $2K$
vertices of $L$ with a given second coordinate $x$, and $|L|\le4t$.
To apply Hall's theorem, suppose that a nonempty subset
$L'\subseteq L$ satisfies $|N(L')|<|L'|$, where $U=N(L')\subseteq R$ is its neighbor set. Let $S\subseteq A_*$
be the set of second coordinates of $L'$. Then
$1\le|U|<|L'|\le\min\{2K|S|,4t\}$.
For each $x\in S$, some $(C,x,i)\in L'$ has at least $r/100$
neighbors by~\eqref{eq:candidates}, all belonging to $U$.
On the other hand, \eqref{eq:reserve-spectrum}, the operator norm
bound and Cauchy--Schwarz imply
\[
 \begin{aligned}
 \frac r{100}|S|
 &\le\sum_{x\in S}d_{G_2}(x,U)
   =\one_S^{\mathsf T}A_{G_2}\one_U\\
 &\le\frac rn|S||U|+\beta\sqrt{|S||U|}
 \le\left(\frac{4rt}{n}+\beta\sqrt{2K}\right)|S|
 =o(r)|S|.
 \end{aligned}
\]
The last equality uses $t=o(n)$, $K=O(\log n)$ and
$\beta\sqrt{2K}/r=O(\log n/n^{1/16})=o(1)$ by \eqref{eq:reserve-spectrum}.
For large $n$ this contradicts $S\ne\varnothing$.
Hall's theorem~\cite[Theorem~16.4]{Bondy2008} therefore gives a
matching covering $L$ and hence distinct choices in $R$ for all
these endpoints. If $L$ is empty, the empty matching suffices.

For a short cycle with attachment vertices $x,y$, let $a,b$ be the two endpoints chosen at $x$ and $c,d$ the two chosen at $y$. The edges $xa,xb,yc,yd$ will be called the
\emph{attachment edges} of this cycle. Lemma~\ref{lem:two-paths}
decomposes the cycle together with these edges into two paths
whose endpoints are $a,b,c,d$. All selected attachment edges are distinct, since no vertex outside
$A_*$ is matched more than once.

Process the long cycles in any fixed order. For each such cycle $C$,
$|D\cap V(C)|\ge|D|+|V(C)|-n>n/12$,
since $|D|\ge n/3$ and $|V(C)|>3n/4$ for large $n$.
To find the two chords of a long cycle $C$, exclude from
$D\cap V(C)$ all vertices already used. These are the attachment
vertices of $A_*$ ($|A_*|\le2t$) and the endpoints of the paths
constructed so far, namely the $4$ endpoints per short cycle of the
attachment edges and the $4$ endpoints per earlier long cycle of its two
chords. Together there are at most $6t$ such vertices.
Thus the remaining set $W\subseteq D\cap V(C)$ satisfies
$|W|\ge n/12-6t\ge n/24$ for large $n$.
By~\eqref{eq:reserve-spectrum}, twice the number of edges of
$G_2[W]$ satisfies
\[
 2e_{G_2}(W)=\one_W^{\mathsf T}A_{G_2}\one_W
 \ge\frac rn|W|^2-\beta|W|>0,
\]
because $(r/n)|W|-\beta\ge r/24-\beta>0$ for large $n$.
Choose one of these edges. Removing its two ends leaves at least
$n/25$ vertices for large $n$, and the same estimate is positive
on the remaining set because $r/25-\beta>0$. We may therefore
choose a second edge with neither endpoint on the first.
Both selected edges are chords of $C$, since their ends lie in
$V(C)$ and they cannot be edges of $C$; indeed, every edge of $C$
belongs to $H\subseteq G_1$, while $G_1$ and $G_2$ are
edge-disjoint. The two selected edges were not used for a short cycle, whose
attachment edges meet $A_*$, or for an earlier long cycle, whose
chord endpoints have been excluded. Lemma~\ref{lem:two-paths} now replaces $C$ and
these two chords by two paths with four new endpoints.

Let $H_1\subseteq G_2$ consist of all the selected attachment edges
and chords, and let $Z$ be the set of endpoints of the $2t$ paths
constructed from the cycles. We have decomposed $H\cup H_1$ into
these paths, whose $4t$ endpoints are distinct, and
$Z\subseteq D\setminus A_*$.
Each use of a vertex of $A_*$ for a short cycle contributes two
edges to $H_1$, and there are at most $K$ such uses at each vertex.
Every vertex of $Z$ has degree one in $H_1$, and every vertex
outside $A_*\cup Z$ has degree zero in $H_1$. Consequently,
$\Delta(H_1)\le2K$ and $D(H_1)=Z$. These two conclusions, together
with the decomposition and $|Z|=4t$, also hold when $t=0$, by
taking $H_1$ to have no edges and $Z=\varnothing$.

The unused edges form the graph $H_2=G-(H\cup H_1)=(G_2-H_1)\dot\cup H_0$ by the
partitions $G=G_1\dot\cup G_2$ and $G_1=H\dot\cup H_0$.
Deleting $H$ preserves all degree parities because $H$ is even,
and deleting $H_1$ changes parity exactly at $Z\subseteq D$.
Therefore,
\begin{equation}\label{eq:remaining-odd}
 D(H_2)=D\setminus Z,\qquad
 |D(H_2)|=|D|-4t=(1/2+o(1))n.
\end{equation}
We check all the other hypotheses of Lemma~\ref{lem:pairing}.
The change in any degree from $G_2$ to $H_2$ is at most
$\Delta(H_1)+\Delta(H_0)=O(\log^2 n)=o(r)$, so
$d_{H_2}(v)=(1+o(1))r$ uniformly over $v$.
Deleting edges does not increase common-neighbor counts. Each new
common neighbor of distinct vertices $u,v$ created by adding
$H_0$ uses an edge of $H_0$ at $u$ or $v$. Hence
$\Delta_2(H_2)\le\Delta_2(G_2)+2\Delta(H_0)=O(\log^2 n)=o(r)$.
Also, $\|A_{H_1}\|\le\Delta(H_1)$ and
$\|A_{H_0}\|\le\Delta(H_0)$ imply
\[
 \left\|A_{H_2}-\frac rn\one\one^{\mathsf T}\right\|
 \le\beta+\Delta(H_1)+\Delta(H_0)=o(r).
\]
Lemma~\ref{lem:normalize} gives $\lambda(H_2)=o(1)$.
Finally, $\log n=o(r)$, $r^3=n^{3/8}=o(n)$, and
\eqref{eq:remaining-odd} gives $|D(H_2)|\ge n/4$ for large $n$.
Thus Lemma~\ref{lem:pairing}, with $c=1/4$, decomposes $H_2$
into exactly $|D(H_2)|/2$ paths.

These paths and the $2t$ paths covering $H\cup H_1$ have disjoint
edge sets and together cover $G$. By~\eqref{eq:remaining-odd},
their total number is
$2t+|D(H_2)|/2=2t+(|D|-4t)/2=\odd(G)/2$.
All the required properties of $G$ hold simultaneously with
probability $1-o(1)$, and the deterministic
Lemma~\ref{lem:pairing} applies on each such outcome.
\end{proof}

\section*{AI Declaration}

ChatGPT was utilized to assist the authors in exploring proof strategies and refining the presentation of this manuscript. All mathematical arguments and proofs in the final manuscript were constructed and verified by the authors, who take full responsibility for all results.

\section*{Acknowledgements}

This research was supported by the National Natural Science Foundation of China (NSFC) grant 12271170 and Science and Technology Commission of Shanghai Municipality grant 22DZ2229014.

\section*{Competing Interests}
The authors declare that there are no competing interests associated with this manuscript.

\section*{Data Availability Statement}
This work is purely theoretical and no experimental datasets were generated or analysed in this research. Data sharing is not applicable to this article.

\end{document}